\documentclass[3p,number]{elsarticle}
\usepackage{amssymb,paralist,mathpazo,charter}
\usepackage[normalem]{ulem}\renewcommand{\underline}[1]{\uline{#1}}
\usepackage[raiselinks=false,colorlinks=true,citecolor=blue,urlcolor=blue,linkcolor=blue,bookmarksopen=true]{hyperref}
\usepackage{xcolor,bookmark} 
\usepackage[frame,arrow,arc,dvips]{xy}
\newtheorem{theorem}{Theorem}
\newtheorem{lemma}[theorem]{Lemma}
\newtheorem{corollary}[theorem]{Corollary}

\newtheorem{proposition}[theorem]{Proposition}
\newproof{proof}{Proof}\let\oldendproof\endproof\def\endproof{\qed\oldendproof} 
\begin{document}

\title{On edge-sets of bicliques in graphs}

\author[uba]{Marina Groshaus}
\ead{groshaus@dc.uba.ar}

\author[sfu]{Pavol Hell}
\ead{pavol@sfu.ca}

\author[wlu]{Juraj Stacho\corref{cor}\fnref{presadd}}
\ead{stacho@cs.toronto.edu}

\address[uba]{Universidad de Buenos Aires, Facultad de Ciencias Exactas y Naturales,
Departamento~de~Computaci\'on, Buenos Aires, Argentina}

\address[sfu]{Simon Fraser University, School of Computing Science, 8888
University Drive, Burnaby, B.C., Canada V5A 1S6}

\address[wlu]{Wilfrid Laurier University, Department of Physics \& Computer
Science, 75~University Ave W, Waterloo, ON N2L 3C5, Canada}

\cortext[cor]{Corresponding author} 

\fntext[presadd]{{\em Present address:} University of Warwick, Mathematics
Institute, Zeeman Building, Coventry, CV4 7AL, United Kingdom}

\begin{abstract}
A {\em biclique} is a maximal induced complete bipartite subgraph of a graph. We
investigate the intersection structure of edge-sets of bicliques in a graph.
Specifically, we study the associated {\em edge-biclique hypergraph} whose
hyperedges are precisely the edge-sets of all bicliques.  We characterize graphs
whose edge-biclique hypergraph is {\em conformal} (i.e., it is the clique
hypergraph of its 2-section) by means of a single forbidden induced obstruction,
the triangular prism.  Using this result, we characterize graphs whose
edge-biclique hypergraph is {\em Helly} and provide a polynomial time
recognition algorithm.  We further study a hereditary version of this property
and show that it also admits polynomial time recognition, and, in fact, is
characterized by a finite set of forbidden induced subgraphs.  We conclude by
describing some interesting properties of the 2-section graph of the
edge-biclique hypergraph.
\end{abstract}

\begin{keyword}
biclique \sep clique graph\sep intersection graph\sep hypergraph\sep
conformal\sep Helly \sep 2-section\smallskip

\noindent{\em AMS classification:} 05C62, 05C75
\end{keyword}

\maketitle

\section{Introduction}
The intersection graph of a collection of sets is defined as follows. The
vertices correspond to the sets, and two vertices are adjacent just if the
corresponding sets intersect. Intersection graphs are a central theme in
algorithmic graph theory because they naturally occur in many applications.
Moreover, they often exhibit elegant structure which allows efficient solution
of many algorithmic problems. Of course, to obtain a meaningful notion, one has
to restrict the type of sets in the collection. In fact, \cite{sur-deux}, every
graph can be obtained as the intersection graph of some collection of sets. By
considering intersections of intervals of the real line, subtrees of a tree, or
arcs on a circle, one obtains interval, chordal, or circular-arc graphs,
respectively. For these classes, a maximum clique or a maximum independent set
can be found in polynomial time \cite{gol}.  We note that one can alternatively define an
interval graph as an intersection graph of connected subgraphs of a path;
similarly intersection graphs of connected subgraphs of a tree produce chordal
graphs, and intersection graphs of connected subgraphs of a cycle produce
circular-arc graphs. More generally, one can consider intersections of
particular subgraphs of arbitrary graphs. This naturally leads to intersections
of edges, cliques, or bicliques of graphs which correspond to line graphs,
clique graphs, and biclique graphs, respectively.

We focus on edge intersections of subgraphs. The edge intersection graph of a
collection of subgraphs is defined in the obvious way, as the intersection graph
of their edge-sets. In hypergraph terminology, this can be defined as the
line graph of the hypergraph whose hyperedges are the edge-sets of the
subgraphs.  We say that subgraphs are edge intersecting if they share at least
one edge of the graph. For instance, the EPT graphs from \cite{ept} are exactly
the edge intersection graphs of paths in trees.  For another example, consider
the double stars of a graph $G$, i.e., the subgraphs formed by the sets of edges
incident to two adjacent vertices. The edge intersection graph of double stars
of $G$ is easily seen to be precisely the square of the line graph of $G$.  In
contrast, if we consider the stars of $G$, i.e., sets of edges incident with
individual vertices, then the edge intersection graph of the stars of $G$ is the
graph $G$ itself \cite{sur-deux}.

In this context, one can study edge intersections of particular subgraphs by
turning the problem into a question about vertex intersections of cliques of an
associated auxiliary graph. In this auxiliary graph, vertices correspond to
edges of the original graph $G$, and two vertices are adjacent just if the
corresponding edges belong to one of the particular subgraphs considered.  In
the language of hypergraphs, this graph is defined as the two-section of the
hypergraph of the edge-sets of the subgraphs.  For instance, in line graphs
vertices are adjacent if and only if the corresponding edges belong to the same
star of $G$.  A similar construction produces the so-called edge-clique graphs
from \cite{ref5} (see also \cite{ref4,ref3,ref2,ref9,ref10}).  Naturally, every
occurence of the particular subgraph in $G$ corresponds to a clique in such
auxiliary graph, and although the converse is generally false, one may obtain
useful information by studying the cliques of the auxiliary graph.

Next, we turn our attention to the Helly property. A collection of sets is said
to have the Helly property if for every subcollection of pairwise intersecting
sets there exists an element that appears in each set of the subcollection. For
instance, any collection of subtrees of a tree has the Helly property. On the
other hand, arcs of a circle or cliques of a graph do not necessarily have the
Helly property. Note that it is, in fact, the Helly property that allows us to
efficiently find a maximum clique in a chordal graph or in a circular-arc graph
(where the Helly property is ``almost'' satisfied \cite{gol}). By comparison, finding a
maximum clique appears to be hard in clique graphs (intersection graphs of
cliques). For a similar reason, recognizing chordal graphs and circular arc
graphs is possible in polynomial time \cite{gol}, whereas it is hard for clique
graphs \cite{clique-graph-np-hard}.

Alternatively, one can impose the Helly property on intersections, and then
study the resulting class of graphs.  For instance, cliques of a graph do not
necessarily satisfy the Helly property, but if we only consider graphs in which
they do, we obtain the class of clique-Helly graphs studied in \cite{clique-helly-prisner}.
In the same way, one can study the classes of neighbourhood-Helly, disc-Helly,
biclique-Helly graphs \cite{ref6}, and also their hereditary counterparts
\cite{ref7, clique-helly}.

\bigskip

In this paper, we investigate the intersections of edge-sets of bicliques.  With
each graph $G$ we associate the {\em edge-biclique hypergraph}, denoted by
${\cal EB}(G)$, defined as follows.  The vertices of ${\cal EB}(G)$ are the
edges of $G$, and the hyperedges of ${\cal EB}(G)$ are the edge-sets of the
bicliques of $G$.  We remark that while for cliques the usual vertex
intersection graphs (i.e., clique graphs and hypergraphs) are the most natural
construct, for bicliques both the vertex and the edge intersection graphs are
natural, and have interesting structure. (See \cite{ref8} for a characterization
of vertex intersection graphs of bicliques.)

The paper is structured as follows. First, in \S 2 we observe some basic
properties of the two-section graph of the edge-biclique hypergraph ${\cal
EB}(G)$. This will allow to prove that ${\cal EB}(G)$ is {\em conformal} (it is
the hypergraph of cliques of its two-section) if and only if $G$ contains no
induced triangular prism. Next, in \S 3 we discuss
the Helly property and prove that ${\cal EB}(G)$ is Helly if and only if the
clique hypergraph of the two-section of ${\cal EB}(G)$ is Helly. This will imply
polynomial time testing for the Helly property on ${\cal EB}(G)$.  In \S 4 we
look at a hereditary version of this property by studying graphs $G$ such that
for every induced subgraph $H$ of $G$, the hypergraph ${\cal EB}(H)$ is Helly.  We show that the
class of such graphs admits a finite forbidden induced subgraph
characterization. This will also yield a polynomial time recognition algorithm
for the class.
In \S 5, we conclude the paper by further discussing properties of the
two-section graph of ${\cal EB}(G)$. In particular, we compare it to the line
graph of $G$, point out some small graphs that are not two-sections of
edge-biclique hypergraphs, and characterize graphs whose every induced subgraph
is the two-section of some edge-biclique hypergraph.  \medskip

\begin{figure}[t!]
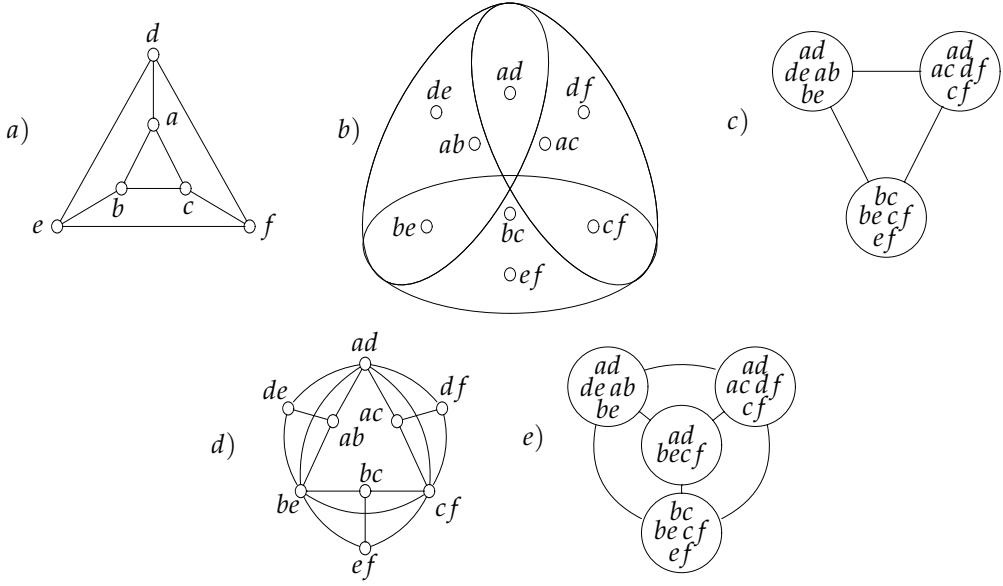

\begin{center}
$\xy/r2pc/:
(0,1)*[o][F]{\phantom{s}}="a";
(0.5,0)*[o][F]{\phantom{s}}="b";
(-0.5,0)*[o][F]{\phantom{s}}="c";
(0,2.1)*[o][F]{\phantom{s}}="d";
(1.5,-0.6)*[o][F]{\phantom{s}}="e";
(-1.5,-0.6)*[o][F]{\phantom{s}}="f";
{\ar@{-} "a";"b"};
{\ar@{-} "a";"c"};
{\ar@{-} "b";"c"};
{\ar@{-} "a";"d"};
{\ar@{-} "b";"e"};
{\ar@{-} "c";"f"};
{\ar@{-} "e";"f"};
{\ar@{-} "d";"f"};
{\ar@{-} "d";"e"};
"a"+(0.3,0.1)*{a};
"b"+(0.05,-0.3)*{c};
"c"+(-0.05,-0.3)*{b};
"d"+(0,0.35)*{d};
"e"+(0.3,0)*{f};
"f"+(-0.3,0)*{e};
(-2.1,0.9)*{a)};
\endxy$
\qquad
$\xy/r2pc/:
(0,1.5)*[o][F]{\phantom{s}}="ad";
(0.55,0.7)*[o][F]{\phantom{s}}="ab";
(-0.55,0.7)*[o][F]{\phantom{s}}="ac";
(0,-0.4)*[o][F]{\phantom{s}}="bc";
(1.3,-0.6)*[o][F]{\phantom{s}}="eb";
(-1.3,-0.6)*[o][F]{\phantom{s}}="fc";
(-1.15,1.2)*[o][F]{\phantom{s}}="df";
(1.15,1.2)*[o][F]{\phantom{s}}="de";
(0,-1.35)*[o][F]{\phantom{s}}="ef";
"ab"+(0.35,0)*{ac};
"bc"+(0.05,-0.3)*{bc};
"ac"+(-0.35,0)*{ab};
"de"+(-0.05,0.35)*{df};
"df"+(0.05,0.35)*{de};
"ef"+(0.35,-0.05)*{ef};
"fc"+(-0.35,0)*{be};
"eb"+(0.35,0)*{cf};
"ad"+(0,0.35)*{ad};
(-2.5,0.9)*{b)};
(-1.35,-0.3);(-0.85,0.7)*\dir{},{\ellipse<58pt,26pt>{-}};
(-0.35,1.7);(-0.85,0.7)*\dir{},{\ellipse<58pt,26pt>{-}};
(1.35,-0.3);(0.85,0.7)*\dir{},{\ellipse<58pt,26pt>{-}};
(0.35,1.7);(0.85,0.7)*\dir{},{\ellipse<58pt,26pt>{-}};
(1,-0.88);(0,-0.88)*\dir{},{\ellipse<55pt,26pt>{-}};
\endxy$
\qquad
$\xy/r2.3pc/:
(-1,1.6)*+++++[o][F]{}="a";
(1,1.6)*+++++[o][F]{}="b";
(0,-0.4)*+++++[o][F]{}="c";
"a"+(0,0.30)*{{{ad}}};
"a"*{{{de}\,{ab}}};
"a"+(0,-0.30)*{{{be}}};
"b"+(0,0.30)*{{{ad}}};
"b"*{{{ac}\,{df}}};
"b"+(0,-0.30)*{{{cf}}};
"c"+(0,0.30)*{{{bc}}};
"c"*{{{be}\,{cf}}};
"c"+(0,-0.30)*{{{ef}}};
{\ar@{-} "a";"b"};
{\ar@{-} "a";"c"};
{\ar@{-} "b";"c"};
(-2.0,0.9)*{c)};
\endxy$\medskip

$\xy/r2pc/:
(0,2.2)*[o][F]{\phantom{s}}="ad";
(-1,0.2)*[o][F]{\phantom{s}}="be";
(1,0.2)*[o][F]{\phantom{s}}="cf";
(-1.2,1.5)*[o][F]{\phantom{s}}="de";
(-0.5,1.3)*[o][F]{\phantom{s}}="ab";
(1.2,1.5)*[o][F]{\phantom{s}}="df";
(0.5,1.3)*[o][F]{\phantom{s}}="ac";
(0,0.2)*[o][F]{\phantom{s}}="bc";
(0,-0.7)*[o][F]{\phantom{s}}="ef";
{\ar@{-}@/_0.3pc/ "ad";"de"};
{\ar@{-}@/^0.3pc/ "ad";"df"};
{\ar@{-} "ad";"ab"};
{\ar@{-} "ad";"ac"};
{\ar@{-} "ab";"be"};
{\ar@{-}@/^0.3pc/ "df";"cf"};
{\ar@{-}@/^0.3pc/ "cf";"ef"};
{\ar@{-}@/_0.3pc/ "de";"be"};
{\ar@{-} "de";"ab"};
{\ar@{-}@/_0.7pc/ "ad";"be"};
{\ar@{-} "be";"bc"};
{\ar@{-}@/_0.3pc/ "be";"ef"};
{\ar@{-}@/_0.7pc/ "be";"cf"};
{\ar@{-} "ac";"cf"};
{\ar@{-} "df";"ac"};
{\ar@{-}@/^0.7pc/ "ad";"cf"};
{\ar@{-} "bc";"cf"};
{\ar@{-} "bc";"ef"};
"ad"+(0,0.35)*{{ad}};
"de"+(-0.2,0.3)*{{de}};
"df"+(0.2,0.3)*{{df}};
"be"+(-0.2,-0.2)*{{be}};
"cf"+(0.3,-0.3)*{{cf}};
"ef"+(0,-0.3)*{{ef}};
"ab"+(0.3,-0.25)*{{ab}};
"ac"+(-0.4,0.1)*{{ac}};
"bc"+(0.1,0.3)*{{bc}};
(-2.2,0.9)*{d)};
\endxy
\qquad
\xy/r2.3pc/:
(-1,1.6)*+++++[o][F]{}="a";
(1,1.6)*+++++[o][F]{}="b";
(0,-0.4)*+++++[o][F]{}="c";
(0,0.8)*+++++[o][F]{}="d";
"d"+(0,-0.15)*{{{be}{cf}}};
"d"+(0,0.15)*{{{ad}}};
"a"+(0,0.30)*{{{ad}}};
"a"*{{{de}\,{ab}}};
"a"+(0,-0.30)*{{{be}}};
"b"+(0,0.30)*{{{ad}}};
"b"*{{{ac}\,{df}}};
"b"+(0,-0.30)*{{{cf}}};
"c"+(0,0.30)*{{{bc}}};
"c"*{{{be}\,{cf}}};
"c"+(0,-0.30)*{{{ef}}};
{\ar@{-}@/^0.7pc/ "a";"b"};
{\ar@{-}@/_1.5pc/ "a";"c"};
{\ar@{-}@/^1.5pc/ "b";"c"};
{\ar@{-} "d";"a"};
{\ar@{-} "d";"b"};
{\ar@{-} "d";"c"};
(-2.0,0.9)*{e)};
\endxy
$\end{center}
\normalsize
\caption{{\em a)} $G$, {\em b)} ${\cal EB}(G)$, {\em c)} the line graph of
${\cal EB}(G)$, {\em d)} $L_G=$ the 2-section of ${\cal EB}(G)$, {\em e)} the clique
graph of $L_G$.\label{fig:5}}
\end{figure}

\section{Notation and Basic Definitions}

A {\em graph} $G=(V,E)$ consists of a vertex set $V$ and a set $E$ of 
edges (unordered pairs from $V$). A {\em hypergraph} ${\cal H}=(V,{\cal E})$
consists of a vertex set $V$ and a set ${\cal E}\subseteq 2^V$ of
hyperedges (subsets of $V$).
For a set $X$ of vertices of a graph $G$, we denote by $G[X]$ the subgraph of $G$ induced by $X$. A
set $X$ is a {\em clique} of $G$ if $G[X]$ is a complete graph and $X$ is
(inclusion-wise) maximal with this property. A set $X$ is a
{\em biclique} of $G$ if $G[X]$ is a complete bipartite graph and $X$ is
(inclusion-wise) maximal with this~property.

For a hypergraph ${\cal H}=(V,{\cal E})$ and a subset ${\cal E'}\subseteq{\cal
E}$, we say that ${\cal H}'=(V,{\cal E}')$ is a {\em partial hypergraph} of~${\cal H}$.  A {\em subhypergraph} of ${\cal H}$ {\em induced} by a set $A\subseteq V$
is the hypergraph ${\cal H}[A]=(A,\{X\cap A~|~X\in{\cal E}\}\setminus\{\emptyset\})$.

To make the presentation clearer, we shall use capital letters $G,H,...$ to
denote graphs and caligraphic letters ${\cal G},{\cal H},...$ to denote
hypergraphs. Similar convention shall be used for graph and hypergraph
operations. In particular, the following operations shall be used throughout the
paper.

Let ${\cal H}=(V,{\cal E})$ be a hypergraph. The \underline{\em dual hypergraph}
of ${\cal H}$, denoted by ${\cal H^*}$, is the hypergraph whose vertex set is
${\cal E}$ and whose hyperedges are $\{{\cal X}_v~|~v\in V\}$ where ${\cal
X}_v=\{X~|~X\in{\cal E}\wedge X\ni v\}$.  In other words, each ${\cal X}_v$
consists of all hyperedges of $H$ that contain $v$.
The \underline{\em 2-section} of ${\cal H}$, denoted by $({\cal H})_2$, is the
graph with vertex set $V$ where two vertices $u,v\in V$ are adjacent if and only
if $u,v\in X$ for some $X\in{\cal E}$.
The \underline{\em line graph} of ${\cal H}$, denoted by $L({\cal H})$, is the
graph with vertex set ${\cal E}$ where $X,X'\in {\cal E}$ are adjacent if and
only if $X\cap X'\neq\emptyset$. Note that $L({\cal H})$ is the 2-section
of the dual hypergraph of $\cal H$.

Let $G=(V,E)$ be a graph.  The \underline{\em line graph} of $G$, denoted by
$L(G)$, is the graph with vertex set $E$ where two edges of $E$ are adjacent if
and only if they share an endpoint in $G$.
The \underline{\em clique hypergraph} of $G$, denoted by ${\cal K}(G)$, is the
hypergraph whose vertex set is $V$ and whose hyperedges are the cliques
of $G$.
The \underline{\em clique graph} of $G$, denoted by $K(G)$, is the graph
whose vertices are the cliques of $G$ where two cliques are adjacent if
and only if they have a vertex in common. In other words, $K(G)$ is the line
graph of the clique hypergraph ${\cal K}(G)$.
The \underline{\em edge-biclique hypergraph} of $G$, denoted by ${\cal EB}(G)$,
is the hypergraph with vertex set is $E$ whose hyperedges are the edge-sets of
the bicliques of~$G$.
The \underline{\em biclique line graph} of $G$, denoted by $L_G$, is the graph
with vertex set $E$ where two edges of $E$ are adjacent if they belong to a
common biclique of $G$.  Note that $L_G$ is the 2-section~of~${\cal EB}(G)$.

For the reader's convenience, we summarize these notions in the following two
tables.\vspace{-1ex}

\begin{center}
\begin{tabular}{c|c|c|c}
${\cal H}=(V,{\cal E})$ & notation & vertices & (hyper)edges\\
\hline& & &\vspace{-2.2ex}\\
dual & ${\cal H}^*$ & hyperedges & hyperedges sharing a common vertex\\
line graph & $L({\cal H})$ & hyperedges & two intersecting hyperedges\\
2-section & $({\cal H})_2$ & vertices & two vertices in a common hyperedge
\end{tabular}\medskip

\begin{tabular}{c|c|c|c}
$G=(V,E)$ & notation & vertices & (hyper)edges\\
\hline& & &\vspace{-2.2ex}\\
line graph & $L(G)$ & edges & two edges sharing a vertex\\
biclique line graph & $L_G$ & edges & two edges in a common biclique\\
clique graph & $K(G)$ & cliques & two intersecting cliques\\
clique hypergraph & ${\cal K}(G)$ & vertices & cliques\\
edge-biclique hypergraph & ${\cal EB}(G)$ & edges & edge-sets of bicliques
\end{tabular}
\end{center}
We also refer the reader to Figure \ref{fig:5} for an illustration of these concepts.

We say that a hypergraph ${\cal H}=(V,{\cal E})$ is {\em reduced} if there are
no hyperedges $X,X'\in{\cal E}$ with $X\subsetneqq X'$.  In other words, a
hypergraph ${\cal H}$ is reduced if every hyperedge of ${\cal H}$ is
inclusion-wise maximal among the hyperedges of ${\cal H}$. 
If ${\cal H}$ is not reduced, then the \underline{\em reduction} of ${\cal H}$
is the partial hypergraph of ${\cal H}$ containing only the inclusion-wise
maximal hyperedges of ${\cal E}$. Note that the reduction of ${\cal H}$ is
always a reduced hypergraph. Also, observe that ${\cal K}(G)$ and ${\cal EB}(G)$
are reduced hypergraphs by definition.

A hypergraph ${\cal H}=(V,{\cal E})$ is {\em Helly} if for every subcollection ${\cal
E}'\subseteq {\cal E}$ satisfying $X\cap X'\neq\emptyset$ for all $X,X'\in {\cal
E'}$, we have $\bigcap_{X\in {\cal E'}}X\neq\emptyset$. 
A hypergraph ${\cal H}$ is {\em conformal} if every clique of the 2-section of
${\cal H}$ is contained in a hyperedge of ${\cal H}$.  In particular, if ${\cal
H}$ is reduced, then ${\cal H}$ is conformal if and only if it is the clique
hypergraph of its 2-section.  Alternatively \cite{berge}, ${\cal H}$ is
conformal if and only if the dual of ${\cal H}$~is~Helly. 

We say that $H$ is a line graph, or a clique graph, or a biclique line graph if,
respectively, $H=L(G)$, or $H=K(G)$, or $H=L_G$, for some $G$.  Note that where
appropriate we shall refer to the vertices of $L_G$ and $K(G)$ as edges and
cliques, respectively, and refer to the hyperedges of ${\cal EB}(G)$ as
bicliques.

As usual, we shall denote by $V(G)$ and $E(G)$ the vertex set respectively the
edge set of a graph $G$. For hypergraphs, we shall not use special notation for
vertices and hyperedges for simplicity.

We emphasize that, in this paper, cliques and bicliques are always maximal, and
they are usually viewed as vertex sets, rather than subgraphs.  For any further
terminology, please consult \cite{berge,west}.

\section{The Conformal Property}

In this section, we characterize graphs $G$ whose edge-biclique hypergraph
${\cal EB}(G)$ is conformal.  We do this by studying the 2-section of ${\cal EB}(G)$.  Recall that
we use $L_G$ to denote the 2-section of ${\cal EB}(G)$ and call this graph the
biclique line graph of $G$.

We start with some useful observations about $L_G$. 
The following is a restatement of the definition.

\begin{proposition}\label{prop:2.1}
If $e=uv$ and $e'=u'v'$ are edges of $G$, then $e$ and $e'$ are adjacent in
$L_G$ if and only if either $u=u'$ and $vv'\not\in E(G)$, or $v=v'$ and
$uu'\not\in E(G)$, or $u,v,u',v'$ induces a four-cycle in $G$.
~\qed
\end{proposition}

\begin{figure}[h!t]
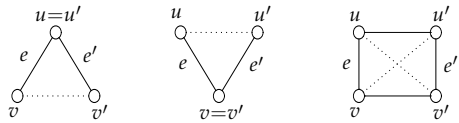

\centering
\vskip -2ex
$\xy/r2pc/:
(0,0.5)*[o][F]{\phantom{s}}="a1";
(-0.6,-0.5)*[o][F]{\phantom{s}}="a2";
(0.6,-0.5)*[o][F]{\phantom{s}}="a3";
"a1"+(0.05,0.3)*{_{u=u'}};
"a2"+(-0.05,-0.25)*{_v};
"a3"+(0.1,-0.25)*{_{v'}};
{\ar@{-}_{e} "a1";"a2"};
{\ar@{-}^{e'} "a1";"a3"};
{\ar@{.} "a2";"a3"};
\endxy$
\qquad
$\xy/r2pc/:
(0,-0.5)*[o][F]{\phantom{s}}="a1";
(-0.6,0.5)*[o][F]{\phantom{s}}="a2";
(0.6,0.5)*[o][F]{\phantom{s}}="a3";
"a1"+(0.05,-0.25)*{_{v=v'}};
"a2"+(-0.05,+0.3)*{_u};
"a3"+(0.1,+0.3)*{_{u'}};
{\ar@{-} "a1";"a2"};
{\ar@{-} "a1";"a3"};
"a1"+(0.6,0.5)*{_{e'}};
"a1"+(-0.55,0.5)*{_{e}};
{\ar@{.} "a2";"a3"};
\endxy$
\qquad
$\xy/r2pc/:
(-0.6,0.5)*[o][F]{\phantom{s}}="a1";
(-0.6,-0.5)*[o][F]{\phantom{s}}="a2";
(0.6,0.5)*[o][F]{\phantom{s}}="a3";
(0.6,-0.5)*[o][F]{\phantom{s}}="a4";
"a1"+(-0.05,0.25)*{_u};
"a2"+(-0.05,-0.25)*{_v};
"a3"+(0.1,0.3)*{_{u'}};
"a4"+(0.1,-0.25)*{_{v'}};
{\ar@{-}_{e} "a1";"a2"}; {\ar@{-}^{e'} "a3";"a4"};
{\ar@{-} "a1";"a3"}; {\ar@{.} "a2";"a3"};
{\ar@{.} "a1";"a4"}; {\ar@{-} "a2";"a4"};
\endxy$
\caption{Adjacent edges in biclique line graphs.\label{fig:2}}
\end{figure}

\noindent In the next lemma and subsequent statements, $\overline P_3$ denotes the complement of the path on 3
vertices.
\begin{lemma}\label{lem:4.1}
If $e_1=ab$ and $e_2=cd$ are edges of $G$ such that $G[a,b,c,d]$ contains a
triangle or an induced $\overline P_3$, then $e_1e_2$ is not an edge of $L_G$.
\end{lemma}
\begin{proof}
Let $ab,cd$ be such edges, and let $H=G[a,b,c,d]$.  First, suppose that $H$
contains a triangle. Without loss of generality, let $a,b,c$ be a triangle of
$H$. If $\{a,b\}\cap \{c,d\}\neq\emptyset$, then $H$ itself is a triangle, and
hence, by Proposition~\ref{prop:2.1}, the edges $e_1=ab$ and $e_2=cd$ are not
adjacent in $L_G$.  So $\{a,b\}\cap \{c,d\}=\emptyset$, but then $H$ is not a
four-cycle implying again that $e_1e_2\not\in E(L_G)$.

Now, suppose that $H$ contains an induced $\overline P_3$.  Without loss of
generality, let $a,b,c$ induce a $\overline P_3$ in $H$ with $ac,bc\not\in
E(G)$. This yields $\{a,b\}\cap \{c,d\}=\emptyset$. Hence, if $e_1e_2\in
E(L_G)$,  it follows from Proposition~\ref{prop:2.1} that this can only be if
$a,b,c,d$ induces a four-cycle. But this contradicts $ac,bc\not\in E(G)$.
\end{proof}

Next, observe that the edge sets of bicliques of $G$ are complete subgraphs of
$L_G$.  In the following, we show that they are, in fact, cliques of $L_G$.

\begin{lemma}\label{thm:3}
The edge-biclique hypergraph of $G$ is a partial hypergraph of the clique
hypergraph of~$L_G$.
\end{lemma}

\begin{proof}
For the proof, we shall show that for every biclique of $G$, its edge set is a
clique in $L_G$.  Consider a biclique $B$ of $G$, and let $C$ denote the edges
of $G[B]$. We shall show that $C$ is a clique of $L_G$. 

Since all edges in the set $C$ belong to a complete bipartite subgraph of $G$,
the set $C$ induces a complete subgraph of $L_G$, as observed above the claim,
by the definition of $L_G$.  Suppose that $C$ is not a clique of $L_G$, that is, there exists an edge $uv=e\not\in C$ such
that $C\cup\{e\}$ is a complete subgraph of $L_G$.  We show that
$G[B\cup\{u,v\}]$ is a complete bipartite graph, which will contradict our
assumption that $B$ is a biclique of $G$.  If $G[B\cup\{u,v\}]$ is not a
complete bipartite graph, then it contains a triangle or an induced $\overline
P_3$ whose at least one vertex is $u$ or $v$.  In particular, if $u,v,a$ induces
in $G$ a triangle or a $\overline P_3$ for some $a\in B$, we let $b$ be any
vertex of $B$ adjacent to $a$ (possibly $b=u$ or $b=v$), and conclude that $ab$
and $uv$ are edges in $C\cup\{e\}$.  This, however, contradicts
Lemma~\ref{lem:4.1}, since then $G[a,b,u,v]$ contains a triangle or an induced
$\overline P_3$.  If $u,a,b$ or $v,a,b$ is a triangle or an induced $\overline
P_3$ in $G$ for $a,b\in B$ where $ab\in E(G)$, we again have edges $ab,uv$ in
$C\cup\{e\}$ contradicting Lemma~\ref{lem:4.1}.  So, we let $a,b$ be
non-adjacent vertices of $B$, and let $c$ be any vertex of $B$ adjacent to $a$
(and hence to $b$).  In particular, $ac$ and $bc$ are edges in $C$, and $u,v$
are not both in $\{a,b,c\}$, since $a,b,c\in B$ and $e\not\in C$.  If exactly
one of $u,v$ is in $\{a,b,c\}$, then we conclude that neither $u,a,b$ nor
$v,a,b$ induces a $\overline P_3$ in $G$, since otherwise we contradict
Lemma~\ref{lem:4.1} for the edges $ac,uv$ or $bc,uv$.  So,
$\{u,v\}\cap\{a,b,c\}=\emptyset$, and we conclude, by Proposition
\ref{prop:2.1}, that both $a,c,u,v$ and $b,c,u,v$ induce a four-cycle in $G$. In
other words, $vc\in E(G)$ if and only if $va,vb\not\in E(G)$ if and only if
$ua,ub\in E(G)$. Hence, both $u,a,b$ and $v,a,b$ do not induce a $\overline P_3$
in $G$. Consequently, $G[B\cup\{u,v\}]$ is a complete bipartite graph, a
contradiction.
\end{proof}

Now, assuming that $G$ contains no induced subgraph isomorphic to the triangular
prism (the graph in Figure \ref{fig:5}a) we show that there are no other cliques
in $L_G$ than the ones arising from bicliques of $G$. 

\begin{lemma}\label{thm:4}
If $G$ contains no induced subgraph isomorphic to the triangular prism, then
the edge-biclique hypergraph of $G$ is equal to the clique hypergraph of $L_G$.
\end{lemma}
\begin{proof}
Assume that $G$ contains no induced subgraph isomorphic to the triangular prism.
By Lemma~\ref{thm:3} it remains to prove that every clique of $L_G$ is the
set of edges of some biclique of $G$.  Consider a clique $C$ of $L_G$, and let $B$ denote the
vertices of $G$ incident to the edges in the set $C$. We show that $B$ is a
biclique of $G$, and $C$ is precisely the set of edges of $G[B]$ which will
prove the claim.

First, we show that the set of edges of $G[B]$ is precisely $C$. Suppose
otherwise, and let $e=uv$ be an edge of $G[B]$ that is not in $C$. Since $u,v\in
B$, we have, by the definition of $B$, edges $au=e^*\in C$ and $bv=e^{**}\in C$.
Clearly, $a\neq v$ and $b\neq u$, since $e\not\in C$. Also, $a\neq b$, because
otherwise $G[a,b,u,v]$ contains a triangle contradicting Lemma \ref{lem:4.1} for
$e^*$ and $e^{**}$ which are adjacent in $L_G$.  Hence, we conclude that the
vertices $a,b,u,v$ induce a four-cycle.  Now, recall that $C$ is a clique of
$L_G$, that is, a maximal complete subgraph of $L_G$. So, since $e\not\in C$,
there must exist an edge $xy=e'\in C$ such that $e$ and $e'$ are not adjacent in
$L_G$. In particular, $e'$ must be adjacent to both $e^*$ and $e^{**}$ in $L_G$.

There are three possibilities.\smallskip

\noindent{\bf Case 1:} $\{x,y\}\cap \{a,b,u,v\}=\emptyset$. Since $e^*$ and $e'$
are adjacent in $L_G$, the vertices $a,u,x,y$ induce a four-cycle in $G$.
Without loss of generality, we may assume that $ux,ay\in E(G)$ and $uy,ax\not\in
E(G)$. Suppose that $vx\in E(G)$. Then it follows that $yb\in E(G)$ and
$xb,yv\not\in E(G)$, since the vertices $b,v,x,y$ induce a four-cycle in $G$.
Thus the vertices
$a,b,u,v,x,y$ induce the triangular prism, a contradiction. Hence, $vx\not\in
E(G)$ and it follows that $vy\not\in E(G)$, since otherwise $u,v,y,x$ induce a
four-cycle in $G$ contradicting the fact that $e$ and $e'$ are not adjacent in
$L_G$. In particular, $G[b,v,x,y]$ contains an induced $\overline{P}_3$. But
then Lemma~\ref{lem:4.1} implies that $e'$ and $e^{**}$ are not adjacent in
$L_G$, a contradiction.\smallskip

\noindent{\bf Case 2:} $y\in \{a,b,u,v\}$ and $x\not\in\{a,b,u,v\}$. First,
suppose that $u=y$. Since $e'$ is adjacent to $e^*$ but not to $e$ in $L_G$, we
have that $ax\not\in E(G)$ and $vx\in E(G)$. Thus $G[b,v,x,y]$ contains a
triangle which, by Lemma~\ref{lem:4.1}, contradicts the fact that $e'$ and
$e^{**}$ are adjacent in $L_G$. Hence, $u\neq y$ and by symmetry, $v\neq y$.
Now, suppose that $a=y$. Again, $xu,xv\not\in E(G)$ since $e'$ is adjacent to
$e^*$ and not adjacent to $e$ in $L_G$, respectively. Thus $G[b,v,x,y]$
contains an induced $\overline{P}_3$ which, again by Lemma \ref{lem:4.1}, leads
to a contradiction. So, $a\neq y$ and by symmetry, $b\neq y$, contradicting
$y\in\{a,b,u,v\}$.\smallskip

\noindent{\bf Case 3:} $x,y\in\{a,b,u,v\}$. This case again leads to a
contradiction, since it is easy to see that all edges of $G[a,b,u,v]$ are
adjacent to $e$ in $L_G$.\medskip

This proves that $C$ is precisely the set of edges of $G[B]$.  Next, we show
that $G[B]$ is a complete bipartite graph.   Suppose otherwise, that is, $G[B]$
contains a triangle or an induced $\overline P_3$.  If $G[B]$ contains a
triangle, then the edges of this triangle are in $C$ but at the same time they
are pairwise not adjacent in $L_G$, contradicting the fact that $C$ is a clique
of $L_G$.  Therefore, there must be vertices $u,v,w$ inducing a $\overline P_3$
in $G[B]$ where $uv\in E(G)$ and $uw,vw\not\in E(G)$.  In particular, $uv$ is an
edge in $C$, and since $w\in B$, there exists, by the definition of $B$, an edge
$zw=e'\in C$. We conclude that $e$ and $e'$ are adjacent in $L_G$, since $C$ is
a clique of $L_G$, which contradicts Lemma~\ref{lem:4.1}, because $u,v,w$ is
an induced $\overline P_3$ in $G[u,v,z,w]$.

We conclude that $G[B]$ is a complete bipartite graph, and hence, there
exists a biclique $B'$ of $G$ such that $B'\supseteq B$.  However, if $C'$ is
the set of edges of $G[B']$, then $C'$ is a complete subgraph of $L_G$ and we
have $C'\supseteq C$. So, we conclude $C'=C$ which yields $B'=B$, and hence, $B$
is a biclique of $G$.

That concludes the proof.
\end{proof}

Note that the assumption in the above theorem cannot be removed since if $G$ is
the triangular prism, the bicliques of $G$ and the cliques of $L_G$ are
different (see Figure \ref{fig:5}). In fact, a stronger statement is true as it
turns out that any graph with an induced triangular prism similarly fails.

We prove this in the following theorem.

\begin{theorem}
For every graph $G$, the edge-biclique hypergraph of $G$ is equal to the clique
hypergraph of $L_G$ if and only if $G$ contains no induced subgraph isomorphic
to the triangular prism.
\end{theorem}

\begin{proof}
The backward direction is proved as Lemma \ref{thm:4}.  For the forward
direction, let $G$ be a graph containing an induced triangular prism on vertices
$a,b,c,d,e,f$ as depicted in Figure \ref{fig:5}. Consider the edges $e_1=ad$,
$e_2=be$, and $e_3=cf$. Note that $e_1,e_2,e_3$ form a triangle in $L_G$. So,
there is a clique $C$ in $L_G$ containing $e_1,e_2,e_3$.  However, we observe
that there is no biclique $B$ in $G$ where $e_1,e_2,e_3$ are edges of
$G[B]$. Indeed, any such $B$ would contain the vertices $a,b,c$ which induce a
triangle in $G$, and hence in $G[B]$, which is impossible. Thus we conclude that
$C$ is a hyperedge of the clique hypergraph of $L_G$ but not a hyperedge of the
edge-biclique hypergraph of $G$.  So, the two hypergraphs are not equal.
\end{proof}

Finally, we notice that ${\cal EB}(G)$ is a reduced hypergraph.  Thus the above
theorem also yields the following corollary which characterizes those graphs $G$
whose edge-biclique hypergraph is conformal.

\begin{corollary}
The edge-biclique hypergraph of a graph $G$ is conformal if and
only if $G$ contains no induced subgraph isomorphic to the triangular
prism.
\end{corollary}

\section{The Helly Property}

We now turn to investigating graphs whose edge sets of bicliques satisfy the
Helly property. 
In particular, we show that the edge-biclique hypergraph of $G$ is Helly if and
only if the clique hypergraph of $L_G$ is Helly.
We start with the following observation.

\begin{lemma}\label{lem:6.1}
If the edge-biclique hypergraph of $G$ is Helly, then $G$ does not contain the
triangular prism as an induced subgraph.
\end{lemma}
\begin{proof}
Let $G$ be a graph such that ${\cal EB}(G)$ is Helly. Suppose that $G$ contains
induced triangular prism on vertices $a,b,c,d,e,f$ as shown in Figure
\ref{fig:5}a. Let $B_1$ be the biclique of $G$ that contains $\{a,b,d,e\}$, let
$B_2$ be the biclique of $G$ that contains $\{b,c,e,f\}$, and let $B_3$ be the
biclique of $G$ that contains $\{a,c,d,f\}$.  Clearly, $c,f\not\in B_1$,
$a,d\not\in B_2$, and $b,e\not\in B_3$. Since ${\cal EB}(G)$ is Helly and the
bicliques $B_1,B_2,B_3$ pairwise intersect in an edge, there must exist an edge
$e=uv$ with $u,v\in B_1\cap B_2\cap B_3$. Clearly, $u\neq a$ since $a\not\in
B_2$.  Similarly, $u\not\in \{a,b,c,d,e,f\}$ and by symmetry we conclude that
$\{u,v\}\cap\{a,b,c,d,e,f\}=\emptyset$. Now, we observe that $u$ is adjacent to
exactly one of $\{a,b\}$, since otherwise $G[B_1]$ contains a triangle or an
induced $\overline{P}_3$, and thus $B_1$ is not a biclique.
Without loss of generality, suppose that $ua\in E(G)$ and $ub\not\in E(G)$. This
implies that $vb\in E(G)$ and $va\not\in E(G)$.  Therefore, $vc\not\in E(G)$,
since otherwise $G[B_2]$ contains a triangle. Thus the vertices $v,a,c$
induce a $\overline{P}_3$ in $G[B_3]$, a contradiction.~\end{proof}

\begin{theorem}\label{thm:5}
The edge-biclique hypergraph of $G$ is Helly if and only if the clique hypergraph
of $L_G$ is Helly.
\end{theorem}
\begin{proof}
By Lemma \ref{thm:3}, the edge sets of bicliques of $G$ are the cliques of
$L_G$.  Hence, if the cliques of $L_G$ satisfy the Helly property, then the edge
sets of bicliques of $G$ must satisfy the Helly property.  Conversely, if ${\cal
EB}(G)$ is Helly, we conclude, by Lemma \ref{lem:6.1},  that $G$ contains no
induced triangular prism.  Hence, by Lemma \ref{thm:4}, the cliques of $L_G$ are
the edge sets of bicliques of $G$. So, if the edge sets of bicliques of $G$
satisfy the Helly property, then the cliques of $L_G$ must satisfy the Helly
property.
\end{proof}

\begin{corollary}
There is a polynomial time algorithm for the recognition of
graphs whose edge-biclique hypergraph is Helly.
\end{corollary}
\begin{proof}
Clearly, the graph $L_G$ can be constructed in polynomial time.
By~\cite{clique-helly}, graphs whose clique hypergraph is Helly can be recognized
in polynomial time. This with Theorem \ref{thm:5} implies the claim. 
\end{proof}

To be more precise, the complexity of the algorithm is $O(|E(G)|^4)$.  This follows
from $O(|E(G)|^2)$ complexity \cite{clique-helly} of recognizing graphs whose clique hypergraph is
Helly. Since we apply this to the graph $L_G$, the total
complexity is $O(|E(L_G)^2|)=O(|E(G)|^4)$. For this note that $L_G$ can have
$O(|E(G)|^2)$ edges, and this is tight, for example, if $G$ is a complete
bipartite graph.  Finally, the construction of the biclique line graph $L_G$
from $G$ can be realized in time $O(|E(G)|^2)$ by a straightforward
implementation. 

We remark that Berge described in \cite{berge} a polynomial time condition for a
family of sets to be Helly.  However, we cannot apply this condition
directly, as a graph can have exponentially many bicliques.

\section{The Hereditary Helly Property}

In this section, we look at a hereditary version of the Helly property for
edge-biclique hypergraphs. This is in a direct analogy with similar classes of
graphs based on the Helly property (e.g., clique-Helly, disk-Helly) whose
corresponding hereditary classes have been considered in the literature (cf.
\cite{ref7}).

We say that a hypergraph ${\cal H}$ is \underline{\em hereditary Helly} if the
reduction of every induced subhypergraph of ${\cal H}$ is Helly.  We require
only reductions of induced subhypergraphs to be Helly so that we obtain a more
general notion also suitable for derived hypergraphs (see below).\smallskip

We study graphs $G$ for which the edge-biclique hypergraph
${\cal EB}(G)$ is hereditary Helly. It can be seen from the definition that
${\cal EB}(G)$ is hereditary Helly if and only if for every induced subgraph
$H$ of $G$, the hypergraph ${\cal EB}(H)$ is Helly. Using this, we describe (in
Theorem \ref{thm:6}) a finite forbidden induced subgraph characterization of
graphs whose edge-biclique hypergraph is hereditary Helly.  \smallskip

{
A {\em $B$-template} is a graph $H$ that consists of a complete bipartite graph
$B$ and three additional vertices $x_1,x_2,x_3$ satisfying one of the following:
\begin{enumerate}[1.]
\item $V(B)=\{1,2,3,z\}$ and $E(B)=\{1z,2z,3z\}$ where for each
$i\in\{1,2,3\}$\smallskip

\begin{compactenum}[(a)]
\item $H[B\setminus\{i\}\cup \{x_i\}]$ is a complete bipartite graph,\smallskip
\item $H[B\cup \{x_i\}]$ is not a complete bipartite graph,
\end{compactenum}

\item $V(B)=\{1,1',2,2',$ $3,3'\}$ and $E(B) = \{11', 12', 13', 21', 22', 23',
31', 32', 33'\}$ where for each
$i\in\{1,2,3\}$\smallskip

\begin{compactenum}[(a)]
\item $H[B\setminus\{i,i'\}\cup \{x_i\}]$ is a complete bipartite graph,
\smallskip
\item $H[B\setminus\{i\}\cup \{x_i\}]$ and $H[B\setminus\{i'\}\cup \{x_i\}]$ are
not complete bipartite graphs.
\end{compactenum}

\end{enumerate}
}

\begin{figure}[h!t]
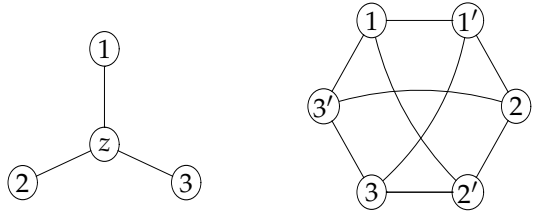

\centering
$\xy/r3pc/:
(0.15,-0.4)*+[o][F-]{2}="A2";
(1,0)*+[o][F-]{\phantom{2}}="Z";
"Z"*{z};
(1.85,-0.4)*+[o][F-]{3}="A3";
(1,1)*+[o][F-]{1}="A1";
{\ar@{-} "Z";"A1"};
{\ar@{-} "Z";"A2"};
{\ar@{-} "Z";"A3"};
\endxy$\qquad\qquad
$\xy/r3pc/:
(0,1.3)*+[o][F-]{1}="A1";
(1,1.3)*+[o][F-]{\phantom{S}}="A1'";
"A1'"*{1'};
(1.5,0.4)*+[o][F-]{2}="A2";
(1,-0.5)*+[o][F-]{\phantom{S}}="A2'";
"A2'"*{2'};
(0,-0.5)*+[o][F-]{3}="A3";
(-0.5,0.4)*+[o][F-]{\phantom{S}}="A3'";
"A3'"*{3'};
{\ar@{-} "A1";"A1'"};
{\ar@{-} "A1'";"A2"};
{\ar@{-}@/^/ "A1'";"A3"+(0.12,0.12)};
{\ar@{-}@/^/ "A2'";"A1"};
{\ar@{-} "A2'";"A2"};
{\ar@{-} "A2'";"A3"};
{\ar@{-}@/_/ "A2";"A3'"};
{\ar@{-} "A2'";"A3"};
{\ar@{-} "A3";"A3'"};
{\ar@{-} "A1";"A3'"};
\endxy$\bigskip
\caption{The graphs $B$ of a $B$-template.\label{fig:x1}}
\end{figure}

{
See Figure \ref{fig:x1} for the two cases of the graph $B$.
All possible $B$-templates are illustrated in Figure \ref{fig:x2}. 
For the proof of our characterization, we shall need the following useful lemma.
}

\begin{figure}
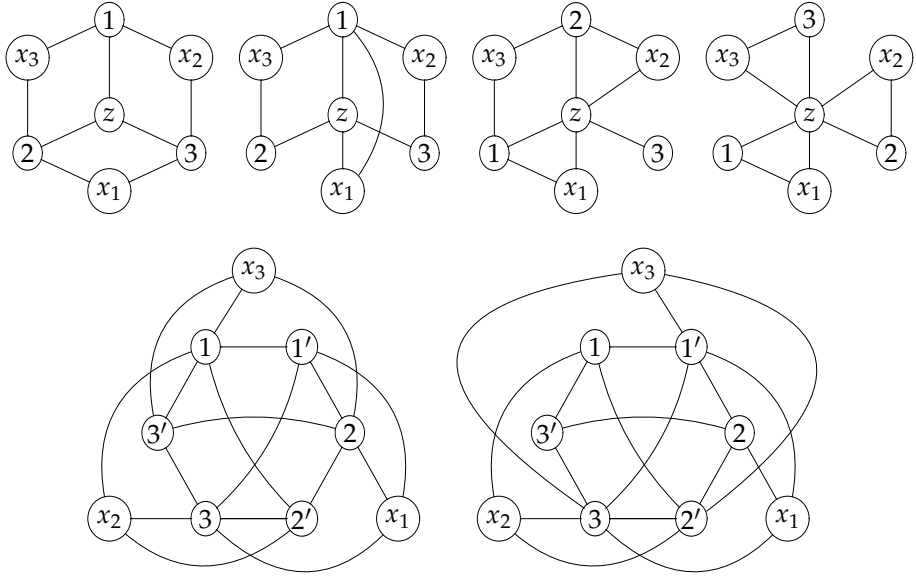

\centering
$\xy/r3pc/:
(0.15,-0.4)*+[o][F-]{2}="A2";
(1,0)*+[o][F-]{\phantom{2}}="Z";
(1.85,-0.4)*+[o][F-]{3}="A3";
(1,1)*+[o][F-]{1}="A1";
"Z"*{z};
(0.15,0.6)*++[o][F-]{\phantom{s}}="x3";
"x3"*{x_3};
(1.85,0.6)*++[o][F-]{\phantom{s}}="x2";
"x2"*{x_2};
(1,-0.8)*++[o][F-]{\phantom{s}}="x1";
"x1"*{x_1};
{\ar@{-} "Z";"A1"};
{\ar@{-} "Z";"A2"};
{\ar@{-} "Z";"A3"};
{\ar@{-} "A2";"x1"};
{\ar@{-} "A3";"x1"};
{\ar@{-} "A1";"x2"};
{\ar@{-} "A3";"x2"};
{\ar@{-} "A1";"x3"};
{\ar@{-} "A2";"x3"};
\endxy$\quad
$\xy/r3pc/:
(0.15,-0.4)*+[o][F-]{2}="A2";
(1,0)*+[o][F-]{\phantom{2}}="Z";
(1.85,-0.4)*+[o][F-]{3}="A3";
(1,1)*+[o][F-]{1}="A1";
"Z"*{z};
(0.15,0.6)*++[o][F-]{\phantom{s}}="x3";
"x3"*{x_3};
(1.85,0.6)*++[o][F-]{\phantom{s}}="x2";
"x2"*{x_2};
(1,-0.8)*++[o][F-]{\phantom{s}}="x1";
"x1"*{x_1};
{\ar@{-} "Z";"A1"};
{\ar@{-} "Z";"A2"};
{\ar@{-} "Z";"A3"};
{\ar@{-} "Z";"x1"};
{\ar@{-}@/^1.0pc/ "A1";"x1"+(0.17,0.17)};
{\ar@{-} "A1";"x2"};
{\ar@{-} "A3";"x2"};
{\ar@{-} "A1";"x3"};
{\ar@{-} "A2";"x3"};
\endxy$\quad
$\xy/r3pc/:
(0.15,-0.4)*+[o][F-]{1}="A1";
(1,0)*+[o][F-]{\phantom{2}}="Z";
(1.85,-0.4)*+[o][F-]{3}="A3";
(1,1)*+[o][F-]{2}="A2";
"Z"*{z};
(0.15,0.6)*++[o][F-]{\phantom{s}}="x3";
"x3"*{x_3};
(1.85,0.6)*++[o][F-]{\phantom{s}}="x2";
"x2"*{x_2};
(1,-0.8)*++[o][F-]{\phantom{s}}="x1";
"x1"*{x_1};
{\ar@{-} "Z";"A1"};
{\ar@{-} "Z";"A2"};
{\ar@{-} "Z";"A3"};
{\ar@{-} "Z";"x1"};
{\ar@{-} "A1";"x1"};
{\ar@{-} "Z";"x2"};
{\ar@{-} "A2";"x2"};
{\ar@{-} "A1";"x3"};
{\ar@{-} "A2";"x3"};
\endxy$\quad
$\xy/r3pc/:
(0.15,-0.4)*+[o][F-]{1}="A1";
(1,0)*+[o][F-]{\phantom{2}}="Z";
(1.85,-0.4)*+[o][F-]{2}="A2";
(1,1)*+[o][F-]{3}="A3";
"Z"*{z};
(0.15,0.6)*++[o][F-]{\phantom{s}}="x3";
"x3"*{x_3};
(1.85,0.6)*++[o][F-]{\phantom{s}}="x2";
"x2"*{x_2};
(1,-0.8)*++[o][F-]{\phantom{s}}="x1";
"x1"*{x_1};
{\ar@{-} "Z";"A1"};
{\ar@{-} "Z";"A2"};
{\ar@{-} "Z";"A3"};
{\ar@{-} "Z";"x1"};
{\ar@{-} "A1";"x1"};
{\ar@{-} "Z";"x2"};
{\ar@{-} "A2";"x2"};
{\ar@{-} "Z";"x3"};
{\ar@{-} "A3";"x3"};
\endxy$\bigskip

$\xy/r3pc/:
(0,1.3)*+[o][F-]{1}="A1";
(1,1.3)*+[o][F-]{\phantom{S}}="A1'";
"A1'"*{1'};
(1.5,0.4)*+[o][F-]{2}="A2";
(1,-0.5)*+[o][F-]{\phantom{S}}="A2'";
"A2'"*{2'};
(0,-0.5)*+[o][F-]{3}="A3";
(-0.5,0.4)*+[o][F-]{\phantom{S}}="A3'";
"A3'"*{3'};
(2,-0.5)*++[o][F-]{\phantom{s}}="x1";
"x1"*{x_1};
(-1,-0.5)*++[o][F-]{\phantom{s}}="x2";
"x2"*{x_2};
(0.5,2.1)*++[o][F-]{\phantom{s}}="x3";
"x3"*{x_3};
{\ar@{-} "A1";"A1'"};
{\ar@{-} "A1'";"A2"};
{\ar@{-}@/^/ "A1'";"A3"+(0.11,0.11)};
{\ar@{-}@/^/ "A2'";"A1"};
{\ar@{-} "A2'";"A2"};
{\ar@{-} "A2'";"A3"};
{\ar@{-}@/_/ "A2";"A3'"};
{\ar@{-} "A2'";"A3"};
{\ar@{-} "A3";"A3'"};
{\ar@{-} "A1";"A3'"};
{\ar@{-}@/_1.5pc/"x1";"A1'"};
{\ar@{-} "x1";"A2"};
{\ar@{-}@/^1.5pc/ "x1";"A3"+(0.11,-0.11)};
{\ar@{-}@/_1.3pc/ "x2";"A2'"-(0.12,0.12)};
{\ar@{-}@/^1.5pc/ "x2";"A1"};
{\ar@{-} "x2";"A3"};
{\ar@{-}@/_1.5pc/ "x3";"A3'"};
{\ar@{-} "x3";"A1"};
{\ar@{-}@/^1.5pc/ "x3";"A2"};
\endxy$\qquad
$\xy/r3pc/:
(0,1.3)*+[o][F-]{1}="A1";
(1,1.3)*+[o][F-]{\phantom{S}}="A1'";
"A1'"*{1'};
(1.5,0.4)*+[o][F-]{2}="A2";
(1,-0.5)*+[o][F-]{\phantom{S}}="A2'";
"A2'"*{2'};
(0,-0.5)*+[o][F-]{3}="A3";
(-0.5,0.4)*+[o][F-]{\phantom{S}}="A3'";
"A3'"*{3'};
(2,-0.5)*++[o][F-]{\phantom{s}}="x1";
"x1"*{x_1};
(-1,-0.5)*++[o][F-]{\phantom{s}}="x2";
"x2"*{x_2};
(0.5,2.1)*++[o][F-]{\phantom{s}}="x3";
"x3"*{x_3};
{\ar@{-} "A1";"A1'"};
{\ar@{-} "A1'";"A2"};
{\ar@{-}@/^/ "A1'";"A3"+(0.11,0.11)};
{\ar@{-}@/^/ "A2'";"A1"};
{\ar@{-} "A2'";"A2"};
{\ar@{-} "A2'";"A3"};
{\ar@{-}@/_/ "A2";"A3'"};
{\ar@{-} "A2'";"A3"};
{\ar@{-} "A3";"A3'"};
{\ar@{-} "A1";"A3'"};
{\ar@{-}@/_1.5pc/"x1";"A1'"};
{\ar@{-} "x1";"A2"};
{\ar@{-}@/^1.5pc/ "x1";"A3"+(0.11,-0.11)};
{\ar@{-}@/_1.3pc/ "x2";"A2'"-(0.12,0.12)};
{\ar@{-}@/^1.5pc/ "x2";"A1"};
{\ar@{-} "x2";"A3"};
{\ar@{-}@/_5pc/ "x3";"A3"+(-0.11,0.11)};
{\ar@{-} "x3";"A1'"};
{\ar@{-}@/^4.5pc/ "x3";"A2'"+(0.16,0.08)};
\endxy$\bigskip

\caption{\label{fig:x2} List of all $B$-templates (excluding the edges between
$x_1,x_2,x_3$).}
\end{figure}

\begin{lemma}\label{lem:1}
Let $G$ be a graph with a vertex $x$ and sets of vertices $B_1\subseteq V(G)$,
$B_2\subseteq V(G)$ such that\vspace{-1ex}
\begin{enumerate}[(i)]
\item $B_1\cap B_2\neq\emptyset$,\vspace{-0.7ex}
\item $B_1\cup B_2$ induces in $G$ a complete bipartite graph, and\vspace{-0.7ex}
\item $B_1\cup \{x\}$ and $B_2\cup \{x\}$ induce in $G$ complete bipartite
graphs. \vspace{-1ex}
\end{enumerate}
Then $B_1\cup B_2\cup\{x\}$ induces in $G$ a complete bipartite graph.
\end{lemma}
\begin{proof}
Suppose that $G[B_1\cup B_2\cup\{x\}]$ is not a complete bipartite
graph. It follows that there must be vertices $a\in B_1\setminus B_2$ and $b\in
B_2\setminus B_1$ such that $x,a,b$ induce in $G$ either a triangle or a
$\overline P_3$. 

Let $z$ be any vertex of $B_1\cap B_2$.
First, suppose that $x,a,b$ induce a triangle in $G$.  Since $G[B_1\cup B_2]$
is a complete bipartite graph, the vertices $a,b,z$ induce neither a triangle
nor a $\overline P_3$, and hence, up to symmetry, we must have $az\not\in E(G)$
and $bz\in E(G)$.  It follows that $xz\in E(G)$, since otherwise $x,a,z$ induce
a $\overline P_3$ contradicting the fact that $G[B_1\cup \{x\}]$ is a complete
bipartite graph.  Thus $x,b,z$ induce a triangle contradicting the fact
that $G[B_2\cup\{x\}]$ is a complete bipartite graph. 

Hence, we conclude that $x,a,b$ induce a $\overline P_3$.  If $ab\in E(G)$, we
may again assume $az\not\in E(G)$ and $bz\in E(G)$. This yields $xz\in E(G)$,
since otherwise $x,b,z$ induce a $\overline P_3$. Thus $x,a,z$ induce a
$\overline P_3$, a contradiction.  Therefore, $ab\not\in E(G)$, and up to
symmetry, we may assume $ax\in E(G)$, and $bx\not\in E(G)$. Yet again, we
conclude $xz\in E(G)$, since otherwise $bz\not\in E(G)$ which implies $az\not\in E(G)$
and $x,a,z$ induce a $\overline P_3$. Consequently, we have $bz\in E(G)$, since
otherwise $x,b,z$ induce a $\overline P_3$. This implies $az\in E(G)$, since
otherwise $a,b,z$ induce a $\overline P_3$.  But now $x,a,z$ induce a
triangle, a contradiction. 
\end{proof}

\begin{theorem}\label{thm:6}
For every graph $G$, the edge-biclique hypergraph of $G$ is hereditary Helly
if and only if $G$ contains no triangular prism and no $B$-template as an induced subgraph.
\end{theorem}
\begin{proof}
For the forward direction, it suffices to verify that the edge-biclique
hypergraph of neither the triangular prism nor any $B$-template is Helly. This
is left for the reader as an excercise.

For the converse, let $G$ be a graph such that ${\cal EB}(G)$ is not Helly, and
let ${\cal B}=\{B_1, B_2,\ldots,B_k\}$ be a minimal family of pairwise edge
intersecting bicliques of $G$ without a common edge. Define
$\mathcal{B}_i=\mathcal{B}\setminus \{B_i\}$ for $i\in\{1,2,3\}$. Since the
family $\cal B$ is minimal, the bicliques in ${\cal B}_i$ have a common edge
$e_i$ for each $i\in\{1,2,3\}$. In addition, $e_i$ is not an edge of $G[B_i]$,
since the bicliques in $\cal B$ have no common edge.  In particular, for each
$i\in\{1,2,3\}$, we have that $e_i$ is an edge of $G[B_j]$ if and only if $j\neq
i$.

There are only three possible cases: the edges $e_1,e_2,e_3$ have a common
vertex, or two of the edges, say $e_2,e_3$ have a common vertex not in $e_1$,
or the three edges share no vertices.\medskip

\noindent {\bf Case 1:} the edges $e_1,e_2,e_3$ have a common vertex $z$.  It
follows that the edges induce a complete bipartite graph with vertices
$\{z,1,2,3\}$ where $e_1=(1,z)$, $e_3=(2,z)$, and $e_3=(3,z)$ as depicted in
Figure \ref{fig:x1}a. By definition, we have $\{z,2,3\}\subseteq B_1$,
$\{z,1,3\}\subseteq B_2$, $\{z,1,2\}\subseteq B_3$, and $1\not\in B_1$,
$2\not\in B_2$, $3\not\in B_3$.  However, $\{z,1,2,3\}$ induces a complete
bipartite graph, and therefore there must exist vertices $x_1\in B_1$, $x_2\in
B_2$, and $x_3\in B_3$ such that none of $\{x_1,z,1,2,3\}$, $\{x_2,z,1,2,3\}$
and $\{x_3,z,1,2,3\}$ induces a complete bipartite graph. In fact, the three
vertices $x_1,x_2,x_3$ must be different.  Suppose otherwise, and say $x_1=x_2$.
Then $\{x_2,z,1,3\}=\{x_1,z,1,3\}$, and hence, $\{x_1,z,1,3\}$, $\{x_1,z,2,3\}$,
and $\{z,1,2,3\}$ induce complete bipartite graphs whereas their union
$\{x_1,z,1,2,3\}$ does not. This contradicts Lemma \ref{lem:1} when applied to
$\{z,1,3\}$, $\{z,2,3\}$ and $x_1$. Hence, the vertices $x_1,x_2,x_3$ are all
distincts yielding a $B$-template $\{x_1,x_2,x_3,z,1,2,3\}$ induced
in $G$.\medskip

{
\noindent {\bf Case 2:} the edges $e_2,e_3$ share a common vertex $z$ not in
$e_1$. It follows that $e_1=(x,y)$, $e_2=(2,z)$, and $e_3=(3,z)$ where $2,3,z$
are distinct vertices, $2\not\in B_2$, $3\not\in B_3$, and $z\not\in\{x,y\}$.
Recall that $e_1\in E(G[B_j])$ if and only if $j\neq 1$.  Thus $x,y\in B_j$ for
all $j\neq 1$, since $e_1=(x,y)$.  This further implies that
$2,3\not\in\{x,y\}$.

Likewise, recall that $e_2\in E(G[B_j])$ for all $j\neq 2$, and $e_3\in
E(G[B_j])$ for all $j\neq 3$. Thus $z\in B_j$ for all $j\in\{1,\ldots,k\}$,
since $z$ is in both $e_2$ and $e_3$. In particular, note that $x,y,z\in B_2$
which implies that one of $xz,yz$ must be an edge, since $B_2$ is a biclique.
By symmetry, assume that $yz\in E(G)$.  We now have two possibilities. 

If $y\not\in B_1$, then we can replace the edge $e_1$ with $e_1'=(y,z)$ to
obtain edges $e_1',e_2,e_3$ satisfying the conditions of Case 1.  On the other
hand, if $y\in B_1$, then both $y$ and $z$ belong to all of the bicliques
$B_1,\ldots,B_k$ and so $yz$ is their common edge.  However, we assume that
there is no such an edge, a contradiction.  \medskip
}

\noindent {\bf Case 3:} the edges $e_1,e_2,e_3$ share no vertices. It is not
difficult to verify that, unless $G$ contains the triangular prism as an induced
subgraph, the edges $e_1,e_2,e_3$ induce a complete bipartite graph with
vertices $\{1,1',2,2',3,3'\}$ where $e_1=(1,1')$, $e_2=(2,2')$, and $e_3=(3,3')$
as depicted in Figure \ref{fig:x1}c. In particular, $\{2,2',3,3'\}\subseteq
B_1$, $\{1,1',3,3'\}\subseteq B_2$, and $\{1,1',2,2'\}\subseteq B_3$.

We show that we may also assume $1,1'\not\in B_1$, $2,2'\not\in B_2$, and
$3,3'\not\in B_3$. Suppose otherwise, say $1\in B_1$.  Then $1'\not\in B_1$,
since $e_1$ is not an edge of $G[B_1]$. If $2\in B_2$, then we can replace $e_1$
with $e_1'=(1',2)$ to obtain edges $e_1',e_2,e_3$ satisfying Case 2. Hence,
$2\not\in B_2$.  Moreover, $3'\not\in B_3$, since otherwise we can replace $e_2$
with $e_2'=(2,3')$ to obtain edges $e_1,e_2',e_3$ satisfying Case 2.
However, now we can replace $e_3$ with $e_3'=(1,3')$ to obtain edges
$e_1,e_2,e_3'$ satisfying Case 2. Therefore, we may conclude  $1\not\in B_1$,
and by symmetry, we have $1,1'\not\in B_1$, $2,2'\not\in B_2$, and $3,3'\not\in
B_3$.

Now, since $\{1,1',2,2',3,3'\}$ induces a complete bipartite graph, there are
again $x_1\in B_1$, $x_2\in B_2$, $x_3\in B_3$ such that none of
$X_1=\{x_1,1,1',2,2',3,3'\}$, $X_2=\{x_1,1,1',2,2',3,3'\}$,
$X_3=\{x_3,1,1',2,2',3,3'\}$ induces a complete bipartite graph. In fact, if
$X_1\setminus \{1\}$ induces a complete bipartite graph, we may replace $B_1$
with a biclique $B_1'$ containing $X_1\setminus \{1\}$ to obtain bicliques
$B_1',B_2,B_3$ satisfying Case~3 for edges $e_1,e_2,e_3$.  However, $1'\in B_1'$
implies that the argument from the above paragraph reduces this situation again
to Case 2. Therefore, we may assume that $X_1\setminus \{1\}$ does induce not a
complete bipartite graph, and by symmetry, none of $X_1\setminus \{1\}$,
$X_1\setminus \{1'\}$, $X_2\setminus \{2\}$, $X_2\setminus \{2'\}$,
$X_3\setminus \{3\}$, $X_3\setminus \{3'\}$ induces a complete bipartite graph.

It remains to observe that the three vertices $x_1,x_2,x_3$ are all distinct.
Indeed, if say $x_1=x_2$, we again contradict Lemma~\ref{lem:1} for
$\{1,1',3,3'\}$, $\{2,2',3,3'\}$ and $x_1$.  Thus
$\{x_1,x_2,x_3,1,1',2,2',3,3'\}$ yields a $B$-template induced in $G$, and this
concludes the proof.  
\end{proof}

Since all forbidden induced subgraphs in the above theorem have at most~9
vertices, we immediately obtain the following consequence.
\begin{corollary}
There is a polynomial time algorithm for the recognition of graphs whose
edge-biclique hypergraph is hereditary Helly.
\end{corollary}

\section{Biclique Line Graphs}

Finally, we discuss some additional interesting properties of biclique line
graphs. A word on notation used in this section. By $K_\ell$ and $\overline
K_\ell$ we denote the complete graph on $\ell$ vertices and its complement,
respectively, and $C_{\ell}$ denotes the cycle on $\ell$ vertices.  Other
special graphs we use are shown in Figure~\ref{fig:4}.

First, we have the following property directly from the definition of $L_G$.
\begin{lemma}\label{lem:11}
If $G$ has no triangle and no induced $C_4$, then
$L_G=L(G)$.~\qed\hspace{-1em}
\end{lemma}

\begin{lemma}\label{lem:11b}
If $L_G$ has no induced $\overline K_3$ and no $K_4$, then $L_G=L(G)$.
\end{lemma}

\begin{proof}
Clearly, if $L_G$ contains no $K_4$, then $G$ contains no induced~$C_4$,
since the edges of any induced $C_4$ in $G$ are always pairwise adjacent in
$L_G$. Also, if $G$ contains a triangle, then $L_G$ contains a $\overline
K_3$, that is, a triple of pairwise non-adjacent vertices, which correspond to
the three edges of the triangle. Consequently, if $L_G$ contains no $K_4$ and
no induced $\overline K_3$, then $G$ has no induced $C_4$ and no triangle.
Hence, $L_G=L(G)$ by Lemma~\ref{lem:11}.
\end{proof}

If we only disallow triangles in $G$, then $L(G)$ becomes a subgraph of
$L_G$, and moreover, we obtain the following characterization.

\begin{theorem} \label{thm:2}
Let $H$ be a graph. Then $H=L_G$ where $G$ is a triangle-free graph
if and only if there exists a set $F\subseteq E(H)$ such that
$H-F=L(G)$ and\vspace{-0.5ex}
\begin{enumerate}[(i)]
\item if $H-F$ contains an induced four-cycle with vertices $a,b,c,d$ and edges
$ab,bc,cd,ad$, then $ac,bd\in F$,\vspace{-1ex}
\item if $ac\in F$, then there exist vertices $b,d$ with $bd\in F$ such that
$a,b,c,d$ induces a four-cycle in $H-F$.
\end{enumerate}
\end{theorem}

\begin{figure}[h!t]
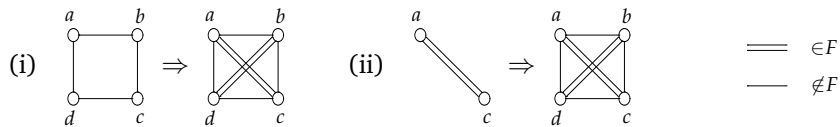

\centering\vskip -2ex
(i)$\quad\xy/r2pc/:
(0,0.5)*[o][F]{\phantom{s}}="a1";
(0,-0.5)*[o][F]{\phantom{s}}="a2";
(1,0.5)*[o][F]{\phantom{s}}="a3";
(1,-0.5)*[o][F]{\phantom{s}}="a4";
{\ar@{-} "a1";"a2"};
{\ar@{-} "a1";"a3"};
{\ar@{-} "a2";"a4"};
{\ar@{-} "a3";"a4"};
"a1"+(-0.05,0.3)*{_a};
"a2"+(-0.05,-0.3)*{_d};
"a3"+(0.05,0.3)*{_b};
"a4"+(0.05,-0.3)*{_c};
\endxy
~\Rightarrow~
\xy/r2pc/:
(0,0.5)*[o][F]{\phantom{s}}="a1";
(0,-0.5)*[o][F]{\phantom{s}}="a2";
(1,0.5)*[o][F]{\phantom{s}}="a3";
(1,-0.5)*[o][F]{\phantom{s}}="a4";
{\ar@{-} "a1";"a2"};
{\ar@{-} "a1";"a3"};
{\ar@{-} "a2";"a4"};
{\ar@{-} "a3";"a4"};
"a1"+(-0.05,0.3)*{_a};
"a2"+(-0.05,-0.3)*{_d};
"a3"+(0.05,0.3)*{_b};
"a4"+(0.05,-0.3)*{_c};
{\ar@{=} "a1";"a4"};
{\ar@{=} "a2";"a3"};
\endxy$
\qquad
(ii)$\quad\xy/r2pc/:
(0,0.5)*[o][F]{\phantom{s}}="a1";
(1,-0.5)*[o][F]{\phantom{s}}="a4";
"a1"+(-0.05,0.3)*{_a};
"a4"+(0.05,-0.3)*{_c};
{\ar@{=} "a1";"a4"};
\endxy
~\Rightarrow~
\xy/r2pc/:
(0,0.5)*[o][F]{\phantom{s}}="a1";
(0,-0.5)*[o][F]{\phantom{s}}="a2";
(1,0.5)*[o][F]{\phantom{s}}="a3";
(1,-0.5)*[o][F]{\phantom{s}}="a4";
{\ar@{-} "a1";"a2"};
{\ar@{-} "a1";"a3"};
{\ar@{-} "a2";"a4"};
{\ar@{-} "a3";"a4"};
"a1"+(-0.05,0.3)*{_a};
"a2"+(-0.05,-0.3)*{_d};
"a3"+(0.05,0.3)*{_b};
"a4"+(0.05,-0.3)*{_c};
{\ar@{=} "a1";"a4"};
{\ar@{=} "a2";"a3"};
\endxy$
\qquad\qquad
$\xy/r2pc/:
(0,-0.3)*{}="a1";
(0.6,-0.3)*{}="a2";
(0,0.3)*{}="a3";
(0.6,0.3)*{}="a4";
{\ar@{-} "a1";"a2"};
{\ar@{=} "a3";"a4"};
"a2"+(0.6,0)*{_{\not\in F}};
"a4"+(0.6,0)*{_{\in F}};
\endxy$
\caption{Conditions (i) and (ii) of Theorem \ref{thm:2}.\label{fig:3}}
\end{figure}

\begin{proof}
Suppose that $H=L_G$ where $G$ is a triangle-free graph.  Since $G$ is
triangle-free,  we have $E(H)\supseteq E(L(G))$. Thus, we choose $F$ to
be the set $F=E(H)\setminus E(L(G))$. Clearly, we have $H-F=L(G)$.  

For the condition (i), let $a,b,c,d$ be an induced four-cycle of $H-F$ with
edges $ab,bc,cd,ad$. Since $H-F$ is the line graph of $G$, it is easy to observe
that $G$ contains a four-cycle whose edges are $a,b,c,d$. Moreover, since $G$ is
triangle-free, this cycle is induced. Thus $a,b,c,d$ induce a complete
subgraph in $H$, and therefore, $ac,bd\in F$.  For the condition (ii), if $ac\in
F$, then $G$ contains an induced four-cycle such that $a,c$ are two opposite
edges of this cycle. Thus, if $b,d$ are the other two edge of this cycle, we
have that $a,b,c,d$ induce a complete subgraph in $H$, and hence, $bd\in F$.

For the other direction, let $F$ be a set of edges of $H$ satisfying the
conditions (i), (ii), and such that $H-F=L(G)$ for some triangle-free graph~$G$.

We show that $H=L_G$.  Suppose that there is an edge $ac\in E(H)$ such that
$ac\not\in E(L_G)$.  Since $G$ is triangle-free, we conclude $ac\not\in
E(L(G))$. Hence, $ac\in F$, and by (ii), there exist $b,d$ such that $a,b,c,d$
induce a four-cycle in $H-F$ and $bd\in F$. Since $H-F$ is a line graph, we
again observe that $G$ contains an induced four-cycle whose edges are $a,b,c,d$.
Thus $ac\in E(L_G)$, a contradiction.  Conversely, suppose that there is
an edge $ac\in E(L_G)$ with $ac\not\in E(H)$.  Since $H-F=L(G)$, we have
$ac\not\in E(L(G))$. Hence, $G$ contains an induced four-cycle whose two
opposite edges are $a,c$.  If $b,d$ are the other two edges of this cycle, we
have that $a,b,c,d$ induce a four-cycle in $L(G)$, and therefore, also in
$H-F$.  Thus, by (i), we have $ac,bd\in F$, and hence, $ac\in E(H)$, a
contradiction.
\end{proof}

Note that the above characterization does not directly imply a polynomial time
algorithm for recognizing biclique line graphs of triangle-free graphs, nor it rules
out such possiblity. It also does not provide any idea about the complexity of
recognizing biclique line graphs of arbitrary graphs.  We remark that the
corresponding problem for line graphs can be solved in polynomial time as
follows from the characterization of \cite{harhol} and from a more general
result of \cite{bieneke}. In these results, polynomial time algorithms are a
consequence of a finite forbidden induced subgraph characterization of line
graphs. This is possible, in particular, because line graphs are closed under
vertex removal. In other words, every induced subgraph of a line graph is again
a line graph. Unfortunately, this is not so for biclique line graphs. In fact,
biclique line graphs are not closed under any of the standard graph operations
(edge, vertex removal, contraction), and hence, it is harder to properly
characterize their structure.  Futhermore, any arbitrary graph can be made to be
an induced subgraph of a biclique line graph as shown in the following claim.

\begin{proposition}\label{prop:13}
For every graph $G$, there exists a graph $G'$ such that $G$ is an induced
subgraph~of~$L_{G'}$.
\end{proposition}
\begin{proof}
We present two constructions. For the first construction, we let $G'$ denote the
graph we obtain by adding to the complement $\overline G$ of $G$ a new vertex
$v$ which we make adjacent to all vertices of $\overline G$. We note that $xy\in
E(G)$ if and only if $xy\not\in E(\overline G)$ if and only if the vertices
corresponding to the edges $xv,yv$ are adjacent in $L_{G'}$. In other words,
the vertices of $L_{G'}$ corresponding to the edges indicent to $v$ induce in
$L_{G'}$ precisely the graph $G$.

For the second construction, we let $G_1$ and $G_2$ denote two disjoint copies
of $G$, and for every vertex $u$ of $G$, we let $u_1$ and $u_2$ denote the
copies of $u$ in $G_1$ and $G_2$, respectively.  We construct the graph $G'$ by
taking the disjoint union of $G_1$ and $G_2$, and adding the edge $u_1u_2$ for
each vertex $u$ of $G$. The graph $G'$ in Figure \ref{fig:5}a illustrates this
construction for $G=K_3$. 
Now, we let $e_u$ denote the vertex of $L_{G'}$ corresponding to the edge
$u_1u_2$ of $G'$.  By Proposition \ref{prop:2.1}, $e_ue_v\in E(L_{G'})$ implies
that $u_1,u_2,v_1,v_2$ is an induced four-cycle of $G'$.  This implies
$u_1v_1\in E(G')$, and hence, $uv\in E(G)$. On the other hand, if $uv\in E(G)$,
then $u_1v_1,u_2v_2\in E(G')$, and hence, $e_ue_v\in E(L_{G'})$, because
$u_1,u_2,v_1,v_2$ induce a four-cycle in $G'$.  Consequently, the subgraph of
$L_{G'}$ induced on $\{e_u~|~u\in V(G)\}$ is precisely the graph $G$.
\end{proof}

\begin{figure}[h!t]
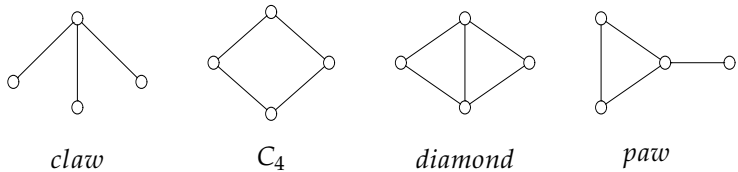

\centering
$\xy/r2pc/:
(0,0.7)*[o][F]{\phantom{s}}="a1";
(-1,-0.3)*[o][F]{\phantom{s}}="a2";
(0,-0.7)*[o][F]{\phantom{s}}="a3";
(1,-0.3)*[o][F]{\phantom{s}}="a4";
{\ar@{-} "a1";"a2"};
{\ar@{-} "a1";"a3"};
{\ar@{-} "a1";"a4"};
(0,-1.5)*{claw};
\endxy$
\qquad
$\xy/r2pc/:
(0,0.8)*[o][F]{\phantom{s}}="a1";
(-0.9,0)*[o][F]{\phantom{s}}="a2";
(0,-0.8)*[o][F]{\phantom{s}}="a3";
(0.9,0)*[o][F]{\phantom{s}}="a4";
{\ar@{-} "a1";"a2"};
{\ar@{-} "a1";"a4"};
{\ar@{-} "a3";"a2"};
{\ar@{-} "a3";"a4"};
(0,-1.5)*{C_4};
\endxy$
\qquad
$\xy/r2pc/:
(0,0.7)*[o][F]{\phantom{s}}="a1";
(-1,0)*[o][F]{\phantom{s}}="a2";
(0,-0.7)*[o][F]{\phantom{s}}="a3";
(1,0)*[o][F]{\phantom{s}}="a4";
{\ar@{-} "a1";"a2"};
{\ar@{-} "a1";"a3"};
{\ar@{-} "a1";"a4"};
{\ar@{-} "a3";"a2"};
{\ar@{-} "a3";"a4"};
(0,-1.5)*{diamond};
\endxy$
\qquad
$\xy/r2pc/:
(-0.7,0.7)*[o][F]{\phantom{s}}="a1";
(1.3,0)*[o][F]{\phantom{s}}="a2";
(-0.7,-0.7)*[o][F]{\phantom{s}}="a3";
(0.3,0)*[o][F]{\phantom{s}}="a4";
{\ar@{-} "a1";"a3"};
{\ar@{-} "a1";"a4"};
{\ar@{-} "a2";"a4"};
{\ar@{-} "a3";"a4"};
(0,-1.5)*{paw};
\endxy$
\caption{The graphs $claw$, $C_4$, $diamond$, and $paw$.\label{fig:4}}
\end{figure}

In order to show that biclique line graphs are not closed under standard
operations, we describe some graphs that are not biclique line graphs.

\begin{proposition}\label{prop:15}
$C_4$, diamond, and claw are not biclique line graphs.
\end{proposition}
\begin{proof}
Clearly, $C_4$ contains no $K_4$ and no triple of pairwise non-adjacent
vertices. Therefore, if $C_4=L_G$ for some graph $G$, we have $L_G=L(G)$
by Lemma \ref{lem:11b}, and $G$ contains no triangle and no induced $C_4$.
However, we must conclude $G=C_4$, since $C_4$ is the only graph whose line
graph is $C_4$, and hence, $G$ contains an induced $C_4$, a contradiction.

Similarly, if $H=diamond$ and $G$ is a graph with $H=L_G$, then $L_G=L(G)$
by Lemma \ref{lem:11b}, since $H$ contains no $K_4$ and no $\overline K_3$. We
must conclude that $G=paw$ (see Figure \ref{fig:4}), which is the only simple
graph whose line graph is $H$. Thus $G$ contains a triangle, a
contradiction.

Finally, let $H=claw$ and $G$ be a graph such that $H=L_G$.  Since $H$ is not
a line graph, we conclude, by Lemma \ref{lem:11b}, that $G$ contains a triangle
or an induced $C_4$. In fact, $G$ contains a triangle, since $H$ has no $K_4$,
and the edges of this triangle form a $\overline K_3$ in $H$. Since there is
only one $\overline K_3$ in $H$, we conclude that $G$ consists of a triangle and
an edge that shares a vertex with every edge of the triangle. However, this is
not possible.
\end{proof}

Now, we see that biclique line graphs are not closed under edge removal, since
$C_4$ is a subgraph of $K_4\cong L_{C_4}$. Similarly, they are not closed under
edge contraction, since $C_5\cong L_{C_5}$ contracts to $C_4$. Moreover, they
are not closed under vertex removal, since, by Proposition~\ref{prop:13}, there
exists a graph $G$ such that $L_G$ contains $C_4$ as an induced
subgraph.\bigskip

Finally, we conclude with the following result.  A graph $H$ is a {\em
hereditary biclique line} graph, if every induced subgraph of $H$ is a biclique
line graph.

\begin{theorem}\label{thm:16}
A graph $H$ is a hereditary biclique line graph if and only if $H$ contains no
induced claw, diamond, or $C_4$.
\end{theorem}
\begin{proof}
Clearly, claw, diamond, and $C_4$ are not biclique line graphs by Proposition
\ref{prop:15}. Hence, it follows that if $H$ is a hereditary biclique line
graph, then $H$ contains no induced claw, diamond, or $C_4$.

Conversely, let $H$ be a graph with no induced claw, diamond, or $C_4$, and
suppose that $H$ is not a hereditary biclique line graph. This implies that $H$
contains an induced subgraph $H'$ that is not a biclique line graph.  Clearly,
$H'$ also contains no induced claw, diamond, or $C_4$. 
In \cite{harhol}, it is shown that if $H'$ does not contain these induced
subgraphs, then it must be the line graph of some triangle-free graph $G$.
Therefore, since $H'$ contains no induced~$C_4$, we can apply Theorem \ref{thm:2}
to $H'$ with $F=\emptyset$ to conclude that $H'$ is also the biclique line graph
of $G$, a contradiction.
\end{proof}

\section*{Acknowledgement}
The authors would like to thank anonymous referees for useful suggestions that
helped improve the presentation of this paper.  The first author was partially
supported by grants UBACyT X456, X143 and ANPCyT PICT 1562 and by CONICET. The
second author was supported by the author's NSERC Discovery Grant.  The third
author acknowledges support from Fondation Sciences Math\'ematiques de Paris and
from Kathie Cameron and Ch\'inh Ho\`ang via their respective NSERC grants.

All three authors also gratefully acknowledge the facilities of IRMACS SFU where
most of this research was done.

\end{document}